\documentclass[reqno,a4paper,11pt]{article}
\usepackage{color}
\usepackage[latin1]{inputenc}
\usepackage[english]{babel}
\usepackage{amsmath,amssymb,amsfonts,amsthm,enumerate}
\usepackage{mathrsfs}
\usepackage[T1]{fontenc}
\usepackage{epsfig}
\usepackage{graphicx}
\usepackage{enumitem}
\usepackage{stmaryrd}
\usepackage[
color,
final] % per togliere showkeys sostituire draft con final
{showkeys}

\usepackage{fullpage}
\newtheorem{theorem}{Theorem}[section]
\newtheorem{lemma}[theorem]{Lemma}
\newtheorem{definition}[theorem]{Definition}

\newcommand{\R}{\mathbb{R}}

\newcommand{\dt}{\mathrm{d}t}

\DeclareMathOperator{\Id}{Id}

\title{$\Gamma$-convergence for a class of action functionals induced by gradients of convex functions}
\author{Luigi \textsc{Ambrosio}\thanks{Scuola Normale Superiore, Pisa. E-mail: \emph{luigi.ambrosio@sns.it}} \and Aymeric \textsc{Baradat}\footnote{Institut Camille Jordan, Lyon. E-mail: \emph{baradat@math.univ-lyon1.fr}} 
	\and Yann \textsc{Brenier}\footnote{\'Ecole Normale Sup\'erieure, Paris. E-mail: \emph{brenier@dma.ens.fr}}}

\begin{document}

	\maketitle
	
	\begin{abstract}
		Given a real function $f$, the rate function for the large deviations of the diffusion process of drift $\nabla f$ given by the Freidlin-Wentzell theorem coincides with the time integral of the energy dissipation for the gradient flow associated with $f$. This paper is concerned with the stability in the hilbertian framework of this common action functional when $f$ varies. More precisely, we show that if $(f_h)_h$ is uniformly $\lambda$-convex for some $\lambda \in \R$ and converges towards $f$ in the sense of Mosco convergence, then the related functionals $\Gamma$-converge in the strong topology of curves.		
	\end{abstract}
	\section{Introduction}
	
	Action functionals of the form 
	$$
	I_f(\gamma):=\int_0^1 \Big\{ |\dot\gamma(t)|^2+|\nabla f|^2(\gamma(t))\Big\} \dt,
	$$
	and the closely related ones (since they differ by a null lagrangian, the term $2f(\gamma(1))-2f(\gamma(0))$)
	\begin{equation}\label{eq:20}
	\int_0^1|\dot\gamma(t)-\nabla f(\gamma(t))|^2\dt,
	\end{equation}
	appear in many areas of Mathematics, for instance in the Freidlin-Wentzell theory of large deviations
	for the SDE $\mathrm{d}X_t^\epsilon=\nabla f(X^\epsilon_t)\dt+\sqrt{\epsilon}\mathrm{d}B_t$ (see for instance \cite{DZ}) or in the variational theory of gradient flows pioneered by
	De Giorgi, where they correspond to the integral form of the energy dissipation (see \cite{AGS}). In this paper, we investigate the stability of the action functionals $I_f$
	with respect to $\Gamma$-convergence of the functions $f$ (actually with respect to the stronger notion of Mosco convergence, see below). More precisely,
	we are concerned with the case when the functions under consideration are $\lambda$-convex 
	and defined in a Hilbert space $H$. In this case, the functional $I_f$ is well defined if we
	understand $\nabla f(x)$ as the element with minimal norm in the subdifferential
	$\partial f(x)$: this choice, very natural in the theory of gradient flows, grants the
	joint lower semicontinuity property of $(x,f)\mapsto|\nabla  f|(x)$ that turns out to be very useful when proving stability of
	gradient flows, see \cite{SS}, \cite{AGS1} and the more recent papers \cite{DFM}, \cite{MPR} where
	emphasis is put on the convergence of the dissipation functionals. In more abstract terms, we are dealing
	with autonomous Lagrangians $L(x,p)=|p|^2+|\nabla f|^2(x)$ that are unbounded and very discontinuous
	with respect to $x$, and this is a source of difficulty in the construction of
	recovery sequences, in the proof of the $\Gamma$-limsup inequality.
	
	Our interest in this problem comes from \cite{ABB}, where we dealt with the
	derivation of the discrete Monge-Amp\`ere equation from the stochastic model of a Brownian
	point cloud, using large deviations and Freidlin-Wentzell theory, along the lines of
	\cite{Brenier}. In that case $H=\R^{Nd}$
	was finite dimensional,
	$$
	f(x):=\max_{\sigma\in\mathfrak{S}_N}\langle x,A^\sigma\rangle,
	$$
	(with $A=(A_1,\ldots,A_N)\in\R^{Nd}$ given and $A^\sigma=(A_{\sigma(1)},\ldots,A_{\sigma(N)})$ for all $\sigma \in \mathfrak{S}_N$, the set of all permutations of $\llbracket 1, N\rrbracket$),
	and the approximating functions $f_\epsilon$ were given by
	$$
	f_\epsilon(t,x)=\epsilon t\log\biggl[\frac 1{N{\rm !} }
	\sum_{\sigma\in\Sigma_N}\exp\bigg(\frac{\langle x,A^\sigma\rangle}{\varepsilon t}\biggr)\biggr].
	$$
	In that case, our proof used some simplifications due to finite dimensionality, and a
	uniform Lipschitz condition. In this paper, building upon some ideas in \cite{ABB}, we
	provide the convergence result in a more general and natural context. For the sake of simplicity,
	unlike \cite{ABB}, we consider only the autonomous case. However it should be possible to
	adapt our proof to the case when time-dependent $\lambda$-convex functions $f(t,\cdot)$ are considered,
	under additional regularity assumptions with respect to $t$, as in \cite{ABB}.
	
	In the infinite-dimensional case, Mosco convergence (see Definition~\ref{defMosco}) is stronger and more appropriate than $\Gamma$-convergence, since 
	it ensures convergence of the resolvent operators (under equi-coercitivity assumptions, the two notions
	are equivalent). Also, since in the infinite-dimensional case, the finiteness domains of the functions can be
	pretty different, the addition of the endpoint condition is an additional source of difficulties, that we handle with
	an interpolation lemma which is very much related to the structure of monotone operators, see Lemma~\ref{lem:fix_endpoints}.
	
	Defining the functionals $\Theta_{f,x_0,x_1}:C([0,1];H)\to [0,\infty]$ by
	\begin{equation}
	\label{eq:def_theta}
	\Theta_{f,x_0,x_1}(\gamma):= \left\{ 
	\begin{aligned}
	&I_f(\gamma) &&\text{if $\gamma\in AC([0,1];H)$, $\gamma(0)=x_0$, $\gamma(1)=x_1$;}\\
	&+\infty &&\text{otherwise,}
	\end{aligned} \right.
	\end{equation}
	our main result reads as follows:
	
	\begin{theorem} \label{thm:main} If $(f_h)_h$ is uniformly $\lambda$-convex for some $\lambda \in \R$, if $f_h\to f$ w.r.t. Mosco convergence, and if
		$$
		\lim_{h\to\infty} x_{h,i}= x_i,\qquad\qquad \sup_h|\nabla f_h|(x_{h,i})<\infty,\qquad i=0,\,1,
		$$
		then $\Theta_{f_h,x_{h,0},x_{h,1}}$ $\Gamma$-converge to $\Theta_{f,x_0,x_1}$ in the $C([0,1];H)$ topology.
	\end{theorem}
	
	As a byproduct, under an additional equi-coercitivity assumption our theorem grants convergence of
	minimal values to minimal values and of minimizers to minimizers.
	Obviously the condition $x_{h,i}\to x_i$ is necessary, and we believe that at least some (possibly more refined) bounds on the
	gradients at the endpoints are necessary as well. If we ask also that $x_{h,i}$ are recovery sequences, i.e. $f_h(x_{h,i})\to f(x_i)$,
	then the result can also be read in terms of the functionals \eqref{eq:20}. 
	
	As a final comment, it would be interesting to investigate this type of convergence results also in a non-Hilbertian
	context, as it happened for the theory of gradient flows. For instance, a natural context would be the space
	of probability measures with finite quadratic moment.  
	Functionals of this form, where $f$ is a 
	constant multiple of  the logarithmic entropy, appear in the so-called entropic regularization of the Wasserstein distance (see \cite{Gentil} and the references therein).
	
	\smallskip
	\noindent
	{\bf Acknowledgements.} We dedicate this paper to Edoardo Vesentini, a great mathematician and a former
	President of the Accademia dei Lincei. As Director of the Scuola Normale, he has been the pioneer of many
	projects that shaped the Scuola Normale for many years to come. The first author acknowledges the support
	of the PRIN 2017 project  ``Gradient flows, Optimal Transport and Metric Measure Structures''.  This work was prepared during the stay of the second author at the Max Planck Institute for Mathematics in the Sciences in Leipzig, that he would like to thank for its hospitality.

	\section{Preliminaries}
	
	Let $H$ be a Hilbert space. For a function $f:H\to (-\infty,\infty]$ we denote by $D(f)$ the finiteness domain of $f$.
	We say that $f$ is $\lambda$-convex if $x\mapsto f(x)-\tfrac{\lambda}{2}|x|^2$ is convex. It is easily seen that
	$\lambda$-convex functions satisfy the perturbed convexity inequality
	$$
	f\bigl((1-t)x+ty\bigr)\leq (1-t)f(x)+t f(y)-\frac\lambda 2 t(1-t)|x-y|^2,\qquad t\in [0,1].
	$$
	
	We denote by $\partial f(x)$ the Gateaux subdifferential of $f$ at $x\in D(f)$, namely the set
	\begin{equation*}
	\partial f(x):=\left\{p\in H:\ \liminf_{t\to 0^+}\frac{f(x+th)-f(x)}{t}\geq t\langle h,p\rangle\,\,\,\forall h\in H\right\}. 
	\end{equation*}
	It is a closed convex set, possibly empty. We denote by $D(\partial f)$ the domain of the subdifferential.
	
	In the case when $f$ is $\lambda$-convex, the monotonicity of difference quotients gives the
	equivalent, non asymptotic definition:
	\begin{equation}\label{eq:def2}
	\partial f(x):=\left\{p\in H:\ f(y)\geq f(x)+\langle y-x,p\rangle+\frac{\lambda}{2}|y-x|^2\,\,\,\forall y\in H\right\}.
	\end{equation}
	For any $x\in D(\partial f)$ we consider the vector $\nabla f(x)$ as the element with minimal
	norm of $\partial f(x)$. We agree that $|\nabla f(x)|=\infty$ if either $x\notin D(f)$ of $x\in D(f)$ and $\partial f(x)=\emptyset$. 
	For $\lambda$-convex functions, relying on \eqref{eq:def2}, it can be easily proved that $\partial f(x)$ is not empty if and only if
	\begin{equation}\label{eq:ags1}
	\sup_{y\neq x}\frac{\bigl[f(x)-f(y)+\frac{\lambda}{2}|x-y|^2\bigr]^+}{|x-y|}<\infty
	\end{equation}
	and that $|\nabla f|(x)$ is precisely equal to the supremum (see for instance Theorem~2.4.9 in \cite{AGS}).
	
	For $\tau>0$ we denote by $f_\tau$ the regularized function
	\begin{equation}\label{eq:Yosida}
	f_\tau(x):=\min_{y\in H}f(y)+\frac{|y-x|^2}{2\tau}
	\end{equation}
	and we denote by $J_\tau=(\Id +\tau\partial f)^{-1}:H\to D(\partial f)$ the so-called resolvent map associating to
	$x$ the minimizer $y$ in \eqref{eq:Yosida}. When $f$ is proper, $\lambda$-convex and lower semicontinuous, existence and uniqueness
	of $J_\tau(x)$ follow by the strict convexity of $y\mapsto f(y)+|y-x|^2/(2\tau)$, as soon as $\tau<-1/\lambda$ when
	$\lambda<0$, and for all $\tau>0$ otherwise (we shall call admissible these values of $\tau$).
	We also use the notation $J_{f,\tau}$ to emphasize the dependence on $f$.
	
	Now we recall a few basic and well-known facts (see for instance \cite{Brezis}, \cite{AGS}), providing for the reader's convenience sketchy proofs.

		\begin{theorem}\label{thm:comprehensive} 
				Assume that $f:H\to (-\infty,\infty]$ is proper, $\lambda$-convex and lower semicontinuous. For all admissible $\tau>0$ one has:
				\begin{enumerate}[label=(\roman*)]
					\item $f_\tau$ is differentiable everywhere, and for all $x \in H$,
					\begin{equation}
					\label{eq:equality_gradients}
					\nabla f_\tau(x) = \frac{x - J_\tau(x)}{\tau} \in \partial f(J_\tau(x)).
					\end{equation}
					\item \label{item:lipschitz} $J_\tau$ is $(1+\lambda\tau)^{-1}$-Lipschitz, and $f_\tau\in C^{1,1}(H)$ with ${\rm Lip}(\nabla f_\tau)\leq 3/\tau$ as soon as there holds $(1+\tau\lambda)^{-1}\leq 2$.
					\item For all $x \in D(\partial f)$,
					\begin{equation}
					\label{eq:equality_gradients_domain}
					\nabla f_\tau(x + \tau \nabla f(x)) = \nabla f(x).
					\end{equation}
					\item The following monotonicity properties hold for all $x \in H$:
					\begin{equation}\label{eq:ags2}
					|\nabla f|(J_\tau (x))\leq |\nabla f_\tau|(x) = \frac{|x-J_\tau (x)|}{\tau}\leq \frac{1}{1+\lambda\tau}|\nabla f|(x). 
					\end{equation}
				\end{enumerate}
		\end{theorem}	
	\begin{proof}
		The inclusion in~\eqref{eq:equality_gradients} follows from performing variations around $J_\tau (x)$ in~\eqref{eq:Yosida}.
		
		Before proving the equality in~\eqref{eq:equality_gradients}, let us prove the Lipschitz property for $J_\tau$ given in~\ref{item:lipschitz}. Recall that the convexity of $g=f-\frac{\lambda}{2} |\cdot|^2$ yields that $\partial f$ is $\lambda$-monotone,
		namely 
		$$
		\langle \xi-\eta,a-b\rangle\geq\lambda |a-b|^2\qquad\forall \xi\in\partial f(a),\,\,\eta\in\partial f(b).
		$$
		Given $x$ and $y$, we apply this property to $a := J_\tau(x)$, $b := J_\tau(y)$, $\xi := (x - J_\tau(x))/\tau$ and $\eta := (y - J_\tau(y))/\tau$. (Thanks to the inclusion in~\eqref{eq:equality_gradients}, we have $\xi \in \partial f(a)$ and $\eta \in \partial f(b)$.) By rearranging the terms, we get
		\begin{equation*}
		\langle x - y , J_\tau(x) - J_\tau(y) \rangle \geq (1 + \lambda \tau) |J_\tau(x) - J_\tau(y)|^2.
		\end{equation*}
		Hence, by the Cauchy-Schwarz inequality, $J_\tau$ is $(1 + \lambda \tau)^{-1}$-Lipschitz.
		
		Let us go back to proving the equality in~\eqref{eq:equality_gradients}. For any $x$ and $z$, one has (using $y=J_\tau(x)$ as an admissible competitor in the definition of $f_\tau(x+z)$)
		\begin{equation*}
		f_\tau(x+z)-f_\tau(x)\leq \frac{|J_\tau(x)-(x+z)|^2}{2\tau}-\frac{|J_\tau(x)-x|^2}{2\tau}=
		\left\langle z,\frac{x-J_\tau(x)}{\tau}\right\rangle+\frac{|z|^2}{2\tau}
		\end{equation*}
		and, reversing the roles of $x$ and $x+z$,
		$$
		f_\tau(x)-f_\tau(x+z)\leq \left\langle -z,\frac{x+z-J_\tau(x+z)}{\tau}\right\rangle+\frac{|z|^2}{2\tau}.
		$$
		These two identities together with the continuity of $J_\tau$ imply that $f_\tau$ is differentiable at $x$ and provides the equality in~\eqref{eq:equality_gradients} and hence the one in~\eqref{eq:ags2}. The Lipschitz property for $\nabla f_\tau$ announced follows directly from this identity and the Lipschitz property for $J_\tau$.
		
		To get~\eqref{eq:equality_gradients_domain}, it suffices to remark that for all $x \in D(\partial f)$, $0$ belongs to the subdifferential of the strictly convex function
		\begin{equation*}
		y \mapsto f(y) + \frac{|x + \tau \nabla f(x)-y|^2}{2\tau}
		\end{equation*}
		at $y = x$. Hence, $x$ is the minimizer of this function, and $J_\tau(x + \tau \nabla f(x)) = x$. Then, we deduce~\eqref{eq:equality_gradients} from~\eqref{eq:equality_gradients_domain}.
		
		The first inequality in~\eqref{eq:ags2} follows from the inclusion in~\eqref{eq:equality_gradients}. In order to prove the second inequality, we perform a variation along the affine curve joining $x$ to $J_\tau(x)$, namely,
		$\gamma(t):=(1-t)x+tJ_\tau(x)$. Since
		\begin{align*}
		f(J_\tau(x))+\frac{1}{2\tau}|x-J_\tau(x)|^2 &\leq f(\gamma(t)) + \frac{1}{2\tau} |x-\gamma(t)|^2\\
		&\leq (1-t) f(x) + t f(J_\tau(x)) + \frac{t}{2\tau} \big(t-\lambda\tau(1-t)\big) |x-J_\tau(x)|^2
		\end{align*}
		for all $t\in [0,1]$, taking the left derivative at $t=1$ gives
		$$
		\left(\frac\lambda 2+\frac 1\tau\right)|x-J_\tau(x)|^2\leq f(x)- f(J_\tau(x)),
		$$
		so that the representation formula \eqref{eq:ags1} for $|\nabla f|(x)$ gives
		$$
		\left(\frac\lambda 2+\frac 1\tau\right)|x-J_\tau(x)|^2\leq |\nabla f|(x)|x-J_\tau(x)|-\frac{\lambda}{2}|x-J_\tau(x)|.
		$$
		By rearranging the terms, this leads to the second inequality in \eqref{eq:ags2}.
	\end{proof}

	Another remarkable property of $|\nabla f|$, for $f$ $\lambda$-convex and lower semicontinuous, 
	is the upper gradient property, namely,
	$$
	f(\gamma(0)),f(\gamma(\delta))< \infty \quad\text{and}\quad |f(\gamma(\delta))-f(\gamma(0))|\leq \int_0^\delta|\nabla f|(\gamma(t))|\dot{\gamma}(t)|\dt
	$$
	for any $\delta>0$ and any absolutely continuous $\gamma:[0,\delta]\to H$ (with the convention $0\times\infty=0$), whenever 
	$\gamma$ is not constant and the integral in the 
	right hand side is finite (see for instance Corollary~2.4.10 in \cite{AGS} for the proof).
	
	\section{A class of action functionals}
	
	For $\delta>0$ and $f:H\to (-\infty,\infty]$ proper, $\lambda$-convex and lower semicontinuous, we consider
	the autonomous functionals $I_f^\delta:C([0,\delta];H)\to [0,\infty]$ defined by
	$$
	I^\delta_f(\gamma):=\int_0^\delta \Big\{|\dot \gamma|^2+|\nabla f|^2(\gamma)\Big\}\dt,
	$$
	set to $+\infty$ on $C([0,\delta];H)\setminus AC([0,\delta];H)$. Notice also that $I_f^\delta(\gamma)<\infty$
	implies $\gamma\in D(\partial f)$ a.e. in $(0,\delta)$.
	
	Identity~\eqref{eq:ags1} ensures the lower semicontinuity of $|\nabla f|$; hence, under a coercitivity
	assumption of the form $\{f\leq t\}$ compact in $H$ for all $t\in\R$, the infimum 
	\begin{equation}\label{eq:5}
	\Gamma_\delta(x_0,x_\delta):=\inf\left\{I^\delta_f(\gamma):\ \gamma(0)=x_0,\,\,\gamma(\delta)=x_\delta\right\}
	\qquad x_0,\,x_\delta\in H
	\end{equation}
	is always attained whenever finite. 
	
	Also, by the Young inequality and the upper gradient property of $|\nabla f|$, one has that 
	$I_f^\delta(\gamma)<\infty$ implies $\gamma(0),\,\gamma(\delta)\in D(f)$ and $2|f(\gamma(\delta))-f(\gamma(0))|\leq I_f^\delta(\gamma)$. 
	The same argument shows that we may add to $I_f^\delta$ a null Lagrangian. Namely, as done in \cite{ABB},
	we can consider the functionals
	$$
	\int_0^\delta |\dot\gamma-\nabla f(\gamma)|^2\dt
	$$
	which differ from $I^\delta_f$ precisely by the term $2f(\gamma(\delta))-2f(\gamma(0))$, whenever $\gamma$ is admissible in
	\eqref{eq:5} with $I_f^\delta(\gamma)<\infty$.
	
	Because of the lack of continuity of $x\mapsto\nabla f(x)$, very little is known in general about the regularity
	of minimizers in \eqref{eq:5}, even when $H$ is finite-dimensional. However, 
	one may use the fact that $I^1_f$ is autonomous to
	perform variations of type $\gamma\mapsto\gamma\circ (\Id+\epsilon\phi)$, $\phi\in C^\infty_c(0,\delta)$, to obtain the Dubois-Reymond
	equation (see for instance \cite{ABO})
	\begin{equation*}	\frac{\mathrm{d}}{\dt}\bigl[|\dot{\gamma}|^2-|\nabla f|^2(\gamma))\bigr]=0\qquad\text{in the sense of distributions in $(0,\delta)$}.
	\end{equation*} 
	It implies Lipschitz regularity of the minimizers when, for instance, $|\nabla f|$ is bounded on bounded sets 
	(an assumption satisfied in \cite{ABB}, but obviously too strong for some applications in infinite dimension).
	
	We will need the following lemma, estimating
	$\Gamma_\delta$ from above, to adjust
	the values of the curves at the endpoints. The heuristic idea is to interpolate on the graph of 
	$f_\tau$ and then read back this interpolation in the original
	variables. This is related to Minty's trick (see \cite{AA} for an extensive use of this idea): a rotation of
	$\pi/4$ maps the graph of the subdifferential onto the graph of an entire $1$-Lipschitz function; here we use
	only slightly tilted variables, of order $\tau$.
	
	\begin{lemma}[Interpolation]\label{lem:fix_endpoints}
		Let $f:H\to (-\infty,\infty]$ be a proper, $\lambda$-convex and lower semicontinuous function and let $\tau>0$
		be such that $(1+\tau\lambda)^{-1}\leq 2$. For all $\delta>0$ and all 
		$x_0\in D(\partial f)$, $x_\delta\in D(\partial f)$, with $\Gamma_\delta$ as in \eqref{eq:5}, one has
		\begin{equation*}
		\Gamma_\delta(x_0,x_\delta)\leq 2\delta \min_{i \in \{0,\delta\}}|\nabla f|^2(x_i)
		+\biggl(\frac {40}\delta+\frac{12\delta}{\tau^2}\biggr)|x_\delta-x_0|^2 +\biggl( 12 \delta + \frac{40\tau^2}{\delta} \biggr) |\nabla f(x_\delta)-\nabla f(x_0)|^2.
		\end{equation*}
	\end{lemma}
	\begin{proof}
		We use Theorem~\ref{thm:comprehensive} to interpolate between
		$x_\delta$ and $x_0$ as follows: set
		$$
		\tilde{\gamma}(t):=\left(1-\frac t\delta\right)(x_0+\tau\nabla f(x_0))+\frac t\delta (x_\delta+\tau\nabla f(x_\delta)),\qquad
		\xi(t):=\nabla f_\tau(\tilde{\gamma}(t)),
		$$
		and
		$$
		\gamma(t):=J_\tau(\tilde \gamma(t)) = \tilde{\gamma}(t)-\tau \xi(t), 
		$$
		where the second equality follows from~\eqref{eq:equality_gradients}.

		Since $\xi(0)=\nabla f_\tau(x_0+\tau\nabla f(x_0))=\nabla f(x_0)$ and a similar
		property holds at time $\delta$, the path $\gamma$ is admissible.
		Let us now estimate the action of the path $\gamma$.
		
		Kinetic term (we use our Lipschitz bound for $\nabla f_\tau$ to deduce that $|\dot \xi(t)| \leq \frac{3}{\tau}|\dot{\tilde \gamma}(t)|$):
		\begin{align*}
			\int_0^\delta |\dot\gamma|^2\dt&\leq 2\int_0^\delta |\dot{\tilde\gamma}|^2\dt+ 2\tau^2\int_0^\delta |\dot\xi|^2\dt\\
			&\leq 20\int_0^\delta|\dot{\tilde\gamma}|^2\dt= \frac{20}{\delta} |(x_\delta+\tau\nabla f(x_\delta))-(x_0+\tau\nabla f(x_0))|^2\\
			&\leq \frac {40}\delta |x_\delta-x_0|^2+\frac{40\tau^2}{\delta}|\nabla f(x_\delta)-\nabla f(x_0)|^2.
		\end{align*}
		
		Gradient term (we use the first inequality in~\eqref{eq:ags2}, our Lipschitz bound for $\nabla f_\tau$, and finally~\eqref{eq:equality_gradients_domain}):
		\begin{align*}
			\int_0^\delta|\nabla f|^2(\gamma)\dt 
			&\leq\int_0^\delta|\nabla f_\tau|^2(\tilde{\gamma})\dt\\
			&\leq \int_0^\delta \left( |\nabla f_\tau|(\tilde \gamma(0)) + \frac{3}{\tau} |\tilde \gamma(t) - \tilde \gamma(0)|)\right)^2 \dt\\
			&\leq \int_0^\delta \left\{ 2 |\nabla f|^2(x_0) + \frac{18}{\tau^2} \frac{t^2}{\delta^2}|(x_\delta+\tau\nabla f(x_\delta))-(x_0+\tau\nabla f(x_0))|^2 \right\} \dt \\
			&\leq 2\delta|\nabla f|^2(x_0) + \frac{6\delta}{\tau^2}|(x_\delta+\tau\nabla f(x_\delta))-(x_0+\tau\nabla f(x_0))|^2 \\
			&\leq  2\delta|\nabla f|^2(x_0)+\frac{12\delta}{\tau^2}|x_\delta - x_0|^2+12\delta |\nabla f(x_\delta)-\nabla f(x_0)|^2.
		\end{align*}
		We get the result by gathering these two estimates, and by remarking that in the second line, we could have controlled $|\nabla f|_\tau(\tilde\gamma(t))$ by its value at time $\delta$ instead of its value at time $0$.
	
	\end{proof}
	
	Choosing $\delta=\tau$, bounding $|\nabla f|(x_i)$, $i = 0,1$ by the max of these two values, and using $ |\nabla f(x_\delta)-\nabla f(x_0)|^2 \leq 4 \max_{i \in \{0,1\}} |\nabla f|^2(x_i)$, we will apply the interpolation lemma in the form
	\begin{equation}\label{eq:14}
	\Gamma_\delta(x_0,x_\delta)\leq \frac {52}\tau|x_\delta-x_0|^2+
	210\tau \max_{i \in \{0,\delta\}}|\nabla f|^2(x_i).
	\end{equation}
	
	\section{Proof of the main result}
	
	In this section, $f_h$, $f$ denote generic proper, $\lambda$-convex and lower semicontinuous functions from $H$ to $(-\infty,\infty]$.
	
	Mosco convergence is a particular case of $\Gamma$-convergence, where the topologies used for the $\limsup$ and the
	$\liminf$ inequalities differ.
	
	\begin{definition}[Mosco convergence]\label{defMosco}
		We say that $f_h$ Mosco converge to $f$ whenever:
		\begin{itemize}
			\item[(a)] for all $x\in H$ there exist $x_h\to x$ strongly with 
			\begin{equation*}
			\limsup_h f_h(x_h)\leq f(x);
			\end{equation*}
			\item[(b)] for all sequences $(x_h)\subset H$ weakly converging to $x$, one has 
			\begin{equation*}
			\liminf_h f_h(x_h)\geq f(x).
			\end{equation*} 
		\end{itemize}
	\end{definition}
	
	It is easy to check that for sequences of $\lambda$-convex functions, 
	Mosco convergence implies the pointwise convergence of $J_{f_h,\tau}$ to $J_{f,\tau}$, contrarily to usual $\Gamma$-convergence.
	Indeed, for $\tau>0$ admissible, (a) grants 
	$$
	\limsup_{h\to\infty}\min_{y\in H}f_h(y)+\frac{|y-x|^2}{2\tau}\leq\min_{y\in H}f(y)+\frac{|y-x|^2}{2\tau},
	$$
	while (b) grants
	$$
	\liminf_{h\to\infty}\min_{y\in H}f_h(y)+\frac{|y-x|^2}{2\tau}\geq\min_{y\in H}f(y)+\frac{|y-x|^2}{2\tau},
	$$
	and the weak convergence of minimizers $y_h$ to the minimizer $y$. Eventually, the convergence of the energies
	together with
	$$
	\liminf_{h\to\infty}f_h(y_h)\geq f(y)\qquad\text{and}\qquad\liminf_{h\to\infty} |y_h-x|^2\geq |y-x|^2
	$$
	grants that both liminf are limits, and that the convergence of $y_h$ is strong.
	
	Recall that given $x_{h,0},\,x_{h,1}\in H$, the functionals $\Theta_{f_h,x_{h,0},x_{h,1}}$ defined in~\eqref{eq:def_theta}, are obtained
	from $I_{f_h}^1$ by adding endpoints constraints.
	$\Theta_{f,x_0,x_1}$ is defined analogously.
	
	We say that $\Theta_{f_h,x_{h,0},x_{h,1}}$ $\Gamma$-converge to $\Theta_{f,x_0,x_1}$ in the $C([0,1];H)$ topology if
	\begin{itemize}
		\item[(a)] for all $\gamma\in C([0,1];H)$ there exist $\gamma_h\in C([0,1];H)$ converging to $\gamma$
		with $$\limsup_{h\to\infty} \Theta_{f_h,x_{h,0},x_{h,1}}(\gamma_h)\leq\Theta_{f,x_0,x_1}(\gamma); $$
		\item[(b)] for all sequences $(\gamma_h)\subset C([0,1];H)$ converging to $\gamma$ one has 
		$$\liminf_{h\to\infty} \Theta_{f_h,x_{h,0},x_{h,1}}(\gamma_h)\geq\Theta_{f,x_0,x_1}(\gamma).$$ 
	\end{itemize}
	
	In connection with the proof of property (a) it is useful to introduce the functional
	$$
	\Gamma-\limsup_{h\to\infty}\Theta_{f_h,x_{h,0},x_{h,1}}(\gamma):=\inf\left\{\limsup_{h\to\infty}\Theta_{f_h,x_{h,0},x_{h,1}}(\gamma_h):
	\gamma_h\to\gamma\right\}
	$$
	so that (a) is equivalent to $\Gamma-\limsup_h\Theta_{f_h,x_{h,0},x_{h,1}}\leq\Theta_{f,x_0,x_1}$. Recall also
	that the $\Gamma-\limsup$ is lower semicontinuous, a property that can be achieved, for instance,
	by a diagonal argument.
	
	\begin{proof}[Proof of Theorem~\ref{thm:main}] It is clear that the endpoint condition passes to the limit with respect to the
		$C([0,1];H)$ topology, since $x_{h,i}$ converge to $x_i$. Also, it is well known that the action functional is lower semicontinuous in
		$C([0,1];H)$. Hence, the $\Gamma$-liminf inequality, namely property (b), follows immediately from
		Fatou's lemma and the variational characterization 
		\eqref{eq:ags1} of $|\nabla f|$. Indeed, for all $y\neq x$ and all sequences $x_h\to x$
		\begin{align*}
			\liminf_{h\to\infty} |\nabla f_h|^2(x_h)\geq\liminf_{h\to\infty}
			\frac{[f_h(x_h)-f_h(y_h)+\frac{\lambda}{2}|x_h-y_h|^2\bigr]^+}{|x_h-y_h|}\geq \frac{[f(x)-f(y)+\frac{\lambda}{2}|x-y|^2\bigr]^+}{|x-y|}.
		\end{align*}
		where $y_h$ is chosen as in (a) of Definition~\ref{defMosco}. Passing to the supremum, we get the inequality
		$\liminf_h|\nabla f_h|(x_h)\geq|\nabla f|(x)$, and this grants the lower semicontinuity of the gradient term in the
		functionals.
		Notice that this part of the proof works also if we assume only that $\Gamma$-$\liminf_h f_h\geq f$, for the strong topology of $H$,
		but the stronger property (namely (b) in Definition~\ref{defMosco}) is necessary because we will need in the next step
		convergence of the resolvents.
		
		So, let us focus on the $\Gamma$-limsup one, property (a). Fix a path $\gamma$ with $\Theta_{f,x_0,x_1}(\gamma)<\infty$, $\tau>0$ (with
		$(1+\tau\lambda^{-1})\leq 2$ if $\lambda<0$) and consider the perturbed paths 
		$\gamma^\tau_h(t)=J_{f_h,\tau}(\gamma(t))$, $\gamma^\tau(t)=J_{f,\tau}(\gamma(t))$; using the 
		$(1+\tau\lambda)^{-1}$-Lipschitz property of the maps $J_{f,\tau}$, the first inequality in
		\eqref{eq:ags2}, the convergence of $\gamma^\tau_h$ to $\gamma^\tau$ and eventually the second inequality in \eqref{eq:ags2} one gets
		\begin{align*}
			\limsup_{h\to\infty}\int_0^1\left\{|\dot\gamma^\tau_h|^2+|\nabla f_h|^2(\gamma^\tau_h)\right\}\dt
			&\leq
			\limsup_{h\to\infty}\int_0^1\left\{(1+\tau\lambda)^{-2}|\dot\gamma|^2+\frac{|\gamma-\gamma^\tau_h|^2}{\tau^2}\right\}\dt\\
			&\leq \int_0^1\left\{(1+\tau\lambda)^{-2} |\dot\gamma|^2 + \frac{|\gamma-\gamma^\tau|^2}{\tau^2}\right\}\dt
			\\&\leq
			(1+\tau\lambda)^{-2}\int_0^1\left\{|\dot\gamma|^2+|\nabla f|^2(\gamma)\right\}\dt.
		\end{align*}
		Also, the convergence of resolvents gives 
		$$
		\lim_{h\to\infty} J_{f_h,\tau}(x_i)=J_{f,\tau}(x_i).
		$$
		Finally, using again the inequalities \eqref{eq:ags2} and once more the convergence of resolvents, we get
		$$
		\limsup_{h\to\infty}\!|\nabla f_h|(J_{f_h,\tau}(x_i))\leq \frac{|J_{f,\tau}(x_i)-x_i|}{\tau}\leq (1+\tau\lambda)^{-1}|\nabla f|(x_i)\leq
		2|\nabla f|(x_i).
		$$
		Since the endpoints have been slightly modified by 
		the composition with $J_{f_h,\tau}$, we argue as follows. Denoting by $S$ an upper bound for
		$|\nabla f_h|(x_{h,i})$ and $2 |\nabla f|(x_i)$, we apply twice the construction of Lemma~\ref{lem:fix_endpoints}, 
		with $\delta=\tau$, to $f_h$ with endpoints $x_{h,i}$, $J_{f_h,\tau}(x_i)$, to extend the curves 
		$\gamma^\tau_h$, still denoted $\gamma_h^\tau$, to the interval $[-\tau,1+\tau]$, in such a way that (we use \eqref{eq:14} in the first inequality, and the second inequality in~\eqref{eq:ags2} in the second one)
		\begin{align*}
		 &\limsup_{h\to\infty}\int_{-\delta}^{1+\delta}\left\{ |\dot\gamma_h^\tau|^2+|\nabla f_h|^2(\gamma_h^\tau)\right\}\dt\\
			&\qquad\leq
			(1+\tau\lambda)^{-2}\int_0^1\left\{|\dot\gamma|^2+|\nabla f|^2(\gamma)\right\}\dt + 420\tau S^2 + \frac{52}{\tau}\Big\{|x_0-J_{f,\tau}(x_0)|^2 + |x_1-J_{f,\tau}(x_1)|^2\Big\}\\
			&\qquad\leq (1+\tau\lambda)^{-2} \left( \int_0^1\left\{ |\dot\gamma|^2 + |\nabla f|^2(\gamma) \right\}\dt + 472\tau S^2 \right)
		\end{align*}
		and the endpoint condition is satisfied at $t=-\tau$ and $t=1+\tau$. The limit of the curves $\gamma^\tau_h$
		in $[-\tau,1+\tau]$, still denoted $\gamma^\tau$, is the one obtained applying the construction of Lemma~\ref{lem:fix_endpoints} with
		$x_i$ and $J_{f,\tau}(x_i)$ in the intervals $[-\tau,0]$ and $[1,1+\tau]$, and which coincides with $J_{f,\tau}(\gamma(t))$ 
		on $[0,1]$.
		
		By a linear rescaling of the curves $\gamma_h^\tau$ and $\gamma^\tau$ to $[0,1]$ we obtain curves
		$\tilde{\gamma}_h^\tau$ converging to $\tilde{\gamma}^\tau$ in $C([0,1];H)$, with $\tilde\gamma^\tau$ convergent to
		$\gamma$ as $\tau\to 0$ and
		\begin{align*}
		\Gamma-\limsup_{h\to\infty}\Theta_{f_h,x_{h,0},x_{h,1}}(\tilde\gamma^\tau)
		&\leq\limsup_{h\to\infty}\Theta_{f_h,x_{h,0},x_{h,1}}(\tilde\gamma_\tau^h)\\
		&\leq (1+O(\tau))\int_0^1\left\{|\dot\gamma|^2+|\nabla f|^2(\gamma)\right\}\dt+O(\tau).
		\end{align*}
		Eventually, the lower semicontinuity of the $\Gamma$-upper limit 
		and the convergence of $\tilde{\gamma}^\tau$ to $\gamma$ provide:
		$$
		\Gamma-\limsup_{h\to\infty}\Theta_{f_h,x_{h,0},x_{h,1}}(\gamma)
		\leq \int_0^1\left\{|\dot\gamma|^2+|\nabla f|^2(\gamma)\right\}\dt.
		$$
	\end{proof}


\vspace{-.8cm}
	\begin{thebibliography}{99}
		
		\bibitem{AA} {\sc G. Alberti, L. Ambrosio:} {\it A geometric approach to monotone functions in ${\bf R}^n$.}
		Matematische Zeitschrift, {\bf 230} (1999),  259--316.
		
		\bibitem{ABO} {\sc L. Ambrosio, G. Buttazzo, O. Ascenzi:} {\it Lipschitz regularity for minimizers of integral functionals
			with highly discontinuous coefficients.} J. Math. Anal. Appl., {\bf 142} (1989), 301--316.
		
		\bibitem{ABB} {\sc L. Ambrosio, A. Baradat, Y. Brenier:} {\it Monge-Amp\`ere gravitation as a $\Gamma$-limit of good rate functions.}
		Preprint, 2020.
		
		\bibitem{AGS} {\sc L. Ambrosio, N. Gigli, G. Savar\'e}:
		{\it Gradient flows in metric spaces and in the space of probability measures.}
		Lectures in Mathematics, ETH Z\"urich, Birkh\"auser (2008).
		
		\bibitem{AGS1} {\sc L. Ambrosio, N. Gigli, G. Savar\'e}: {\it Calculus and heat flow in metric measure spaces and applications to
			spaces with Ricci bounds from below.} Inventiones Mathematicae, {\bf 195} (2014), 289--391.
		
		\bibitem{Brenier} {\sc Y. Brenier:} 
		{\it A double large deviation principle for Monge-Amp\`ere gravitation.} 
		Bull. Inst. Math. Acad. Sin. (N.S.), {\bf 11} (2016), 23--41.
		
		\bibitem{Brezis} {\sc H. Brezis:} {\it Op\'erateurs maximaux monotones et semi-groupes de contractions dans les espaces 
			de Hilbert.} North-Holland Publishing Co., Amsterdam, 1973.
		
		\bibitem{Gentil} {\sc G. Clerc, I. Gentil, G. Conforti:}
		{\it On the variational interpretation of local logarithmic Sobolev inequalities.}
		Preprint, 2021.
		
		\bibitem{DZ} {\sc A. Dembo, O. Zeitouni:} {\it Large deviation techniques and applications.} Applications of
		Mathemmatics {\bf 38}, Springer, 1998. 
		
		\bibitem{DFM} {\sc P. Dondl, T. Frenzel, A. Mielke:}
{\it A gradient system with a wiggly energy and relaxed EDP-convergence.}
 ESAIM Control Optim. Calc. Var., {\bf 25} (2019), paper no. 68, 45pp.
		
		\bibitem{MPR} {\sc A. Mielke, M.A. Peletier, D.R.M. Renger:}
 {\it On the relation between gradient flows and the large-deviation
              principle, with applications to {M}arkov chains and diffusion.}
  Potential Anal., {\bf 41} (2014), 1293--1327.
		
		\bibitem{SS} {\sc E. Sandier, S. Serfaty:} {\it $\Gamma$-Convergence of Gradient Flows with Applications to Ginzburg-Landau.}
		Comm. Pure Appl. Math., {\bf 57} (2004), 1627--1672.
		
				
	\end{thebibliography}
\end{document}